\documentclass{amsart}
\usepackage{amsmath,amsthm,amsfonts,amssymb,graphicx}
\usepackage{hyperref}
\usepackage{url}
\usepackage{tikz-cd}
\usepackage{bbm}
\usepackage[OT2,T1]{fontenc}
\usepackage[skip=0pt]{caption}
\usepackage[group-separator={,}]{siunitx}
\usepackage{placeins}

\newtheorem{theorem}{Theorem}[section]

\newtheorem{proposition}[theorem]{Proposition}

\DeclareMathOperator{\rk}{rank}
\DeclareMathOperator{\Sel}{Sel}
\DeclareSymbolFont{cyrletters}{OT2}{wncyr}{m}{n}
\DeclareMathSymbol{\Sha}{\mathalpha}{cyrletters}{"58}
\newcommand{\leg}[2]{{\left({\frac{#1}{#2}}\right)}}
\newcommand{\Z}{\mathbb{Z}}

\newcommand{\Q}{\mathbb{Q}}
\newcommand{\E}{\mathbb{E}}
\newcommand{\1}{\mathbbm{1}}
\renewcommand{\P}{\mathbb{P}}

\let\originalleft\left
\let\originalright\right
\let\originalhat\hat
\renewcommand{\left}{\mathopen{}\mathclose\bgroup\originalleft}
\renewcommand{\right}{\aftergroup\egroup\originalright}
\renewcommand{\hat}[1]{\smash{\originalhat{#1}}}

\allowdisplaybreaks

\author{Stephanie Chan}
\address{Department of Mathematics, University College London, Gower Street, London, WC1E 6BT, United Kingdom}
\email{stephanie.chan.16@ucl.ac.uk}

\author{Jeroen Hanselman}
\address{Institute of Pure Mathematics, Ulm University, Helmholtzstrasse 18, 89081 Ulm, Germany}
\email{jeroen.hanselman@uni-ulm.de}

\author{Wanlin Li}
\address{Department of Mathematics,	University of Wisconsin, Madison, WI 53706, USA}
\email{wanlin@math.wisc.edu}

\title[elliptic curves with $\Z/2\Z \times \Z/8\Z$-torsion]{Ranks, $2$-Selmer groups, and Tamagawa numbers of elliptic curves with $\Z/2\Z \times \Z/8\Z$-torsion}

\begin{document}

\begin{abstract}
In 2016, Balakrishnan--Ho--Kaplan--Spicer--Stein--Weigandt \cite{Main} produced a database of elliptic curves over $\mathbb{Q}$ ordered by height in which they computed the rank, the size of the $2$-Selmer group, and other arithmetic invariants. They observed that after a certain point, the average rank seemed to decrease as the height increased. Here we consider the family of elliptic curves over $\mathbb{Q}$ whose rational torsion subgroup is isomorphic to $\mathbb{Z}/2\mathbb{Z} \times \mathbb{Z}/8\mathbb{Z}$. Conditional on GRH and BSD, we compute the rank of $92\%$ of the $\num{202461}$ curves with parameter height less than $10^3$. We also compute the size of the $2$-Selmer group and the Tamagawa product, and prove that their averages tend to infinity for this family.
\end{abstract}
\maketitle

\section{Introduction}
Let $E$ be an elliptic curve over $\mathbb{Q}$. After a suitable choice of isomorphism, we can always express such a curve in its short Weierstrass form:
\[ E:y^2=x^3 + a_4 x + a_6 \]
with $a_4,a_6\in\mathbb{Z}$. Using this description, we define the naive height of the curve $E$ as
$h(E):= \max \{ 4|a_4|^3, 27a_6^2 \}$.

In \cite{Main}, the authors created an exhaustive database of isomorphism classes of elliptic curves with naive height up to $2.7 \cdot 10^{10}$, which contained a total of $\num{238764310}$ curves. For each elliptic curve in this database, they computed the minimal model, the torsion subgroup, the conductor, the Tamagawa product, the rank, and the size of the $2$-Selmer group. They plotted the average rank of the curves up to a certain height. Initially the average rank seemed to be an increasing function, but around a naive height of $10^9$, they observed a turnaround point, where the average rank seemed to start decreasing as the height was increasing.

In this database however, there were no elliptic curves recorded with rational torsion subgroup isomorphic to $\mathbb{Z} / 2 \mathbb{Z} \times \mathbb{Z} / 8 \mathbb{Z}$, which is the largest possible torsion subgroup for elliptic curves over $\mathbb{Q}$. The curve with minimal naive height that has such a torsion group has Weierstrass form $y^2 = x^3 - 1386747x + 368636886$ and its naive height is $10667230914617018892 \approx 1.07 \cdot 10^{19}$.

In this paper, we describe a similar database for the family of elliptic curves over $\mathbb{Q}$ whose rational torsion subgroup is isomorphic to $\mathbb{Z} / 2 \mathbb{Z} \times \mathbb{Z} / 8 \mathbb{Z}$.
We can parametrize this family in the following way:
 \[\mathcal{F}:=\left\{E:y^2 = x(x+1)(x+u^4)\ \middle| \ u = \frac{2t}{t^2-1},\quad t\in\Q\setminus\{0, 1\}\right\}.\]
We call $t$ the parameter of the curve and write $t=a/b$ for coprime integers $a$, $b$.
This particular parametrization was provided by Bartosz Naskr\k{e}cki, resulting from ideas in \cite{Naskrecki}.
The family inherits a height function from its parametrization. For any $E\in\mathcal{F}$, we define the parameter height $H(E):=\max\{|a|,|b|\}$. For each isomorphism class of curves in this family, we will only consider the model in $\mathcal{F}$ for which $H$ is minimal. From now on, we will call the family of curves represented by elements of $\mathcal{F}$ the $(2,8)$-torsion family.

We use the parameter height, as it makes it easier to enumerate and compare curves in our family. The naive height of the curves in our family is very large, as could already be seen in the example mentioned above. We prove in Section~\ref{sec:properties} that
\[0.559\cdot h(E)^{1/48}<H(E)< 0.672\cdot h(E)^{1/48}.\]
We also show that the parameter height controls the size of the conductor $N(E)$:
\[N(E)< 1.161\cdot H(E)^{10}.\] 
From now on, we will use the term \emph{height} to refer to the parameter height.

There are several reasons to consider the $(2,8)$-torsion family. First, based on the relation between the parameter height and the naive height, restricting to this family allows us to quickly see curves of large naive height. Another advantage is that the existence of the rational torsion structure makes it easier to carry out $2$-descent.

To provide an example, the $2000$th curve in our database has parameter $t=98/99$, naive height $6.39\cdot 10^{107}$ and conductor $6.65\cdot 10^{17}$. It would be more difficult to determine the rank for a curve of similar size without any special structure, and currently it would not be feasible to carry out such calculations in bulk.

In our family, we enumerated all $\num{202461}$ isomorphism classes of curves with height less than $1000$. The average rank function seems to achieve its maximum at height $24$, at the $121$st curve, where the average rank peaks at $0.744$. Among these, we determined the rank for $\num{186719}$ classes, conditional on GRH and BSD.

This particular family of elliptic curves was also studied in \cite{SUMSRI2006} and \cite{SUMSRI2007}. In \cite{SUMSRI2006}, the authors were in search of rank $4$ curves, but were unable to find any. To date, no rank $4$ curve has yet been found in this family. In \cite{SUMSRI2007}, the authors obtained statistical results on the $2$-Selmer group, similar to our data in Section~\ref{subsec:avselmer}.

\subsection*{Main Results}
We found that curves with height up to $100$ in the $(2,8)$-torsion family has average rank $0.626$ (Figure~\ref{fig:H100} in Section~\ref{subsec:avrank}) and with height up to $1000$ have average rank between $0.508$ and $0.663$ (Figure~\ref{fig:H1000} in Section~\ref{subsec:avrank}).
The first curves in the $(2,8)$-torsion family with given rank $r$ are
\begin{align*}
r=0:\ &y^2 = x^3 - 1386747x + 368636886 &&(t=1/2),\\
r=1:\ &y^2 = x^3 - 64052311707x + 6090910426477494 &&(t=1/4),\\
r=2:\ &y^2 = x^3 - 42884506779312987x + 3379377560795274084396534 &&(t=5/8),\\
r=3:\ &y^2 = x^3 - 20406728559954500484507x \\[-1pt]
&\omit\hfill+ 1121060630379489735235148874483894 &&(t=12/17).
\end{align*}
We found that no rank $4$ curves can exist with height below $1000$.

The curve with rank $3$ with the greatest height found in our database has parameter $t=841/1018$; its global minimal model is as follows:
\[\begin{split}
&y^2 + xy = x^3 - 1537294523297507321569249472559902413559297102550x 
\\&+ 733636624633313284630814852522791055015138014738294124679165680060100132 .\end{split}\]
This curve was found when we tried to compute the $2$-Selmer rank of curves beyond height $1000$.
Currently, the curve with maximal height on the list of elliptic curves with high rank maintained by Dujella \cite{Dujella} has parameter $352/1017$.

The average size of the $2$-Selmer group seems to be increasing rather slowly, but steadily. We prove the following theorem, which is an analogue of a result by Lemke-Oliver and Klagsbrun for the family of elliptic curves with $2$-torsion \cite{KLO}.

\newtheorem*{avsel}{Theorem~\ref{thm:avsel}}
\begin{avsel}
The average size of the $2$-Selmer group tends to infinity in the $(2,8)$-torsion family.
\end{avsel}

Similarly, observing the data on the average Tamagawa product suggested the following theorem that we prove in Section~\ref{sec:avtam}:

\newtheorem*{avtam}{Theorem~\ref{thm:avtam}}
\begin{avtam}
The average Tamagawa product in the $(2,8)$-torsion family up to height $N$ has order of magnitude $(\log N)^{33}$.
\end{avtam}

\subsection*{Outline of the paper} In Section~\ref{sec:properties}, we provide some properties of the $(2,8)$-torsion family related to our parametrization. In Section~\ref{sec:background}, we recall general results and conjectures related to ranks of elliptic curves. In Section~\ref{sec:computation}, we discuss the computational methods we use. Section~\ref{sec:data} contains the data we obtained and our analysis of the data. In Section~\ref{sec:proof}, we prove that the average Tamagawa product and the average size of the $2$-Selmer group tends to infinity for this family.

\subsection*{Acknowledgements}
The authors would like to thank Jennifer Balakrishnan for suggesting the topic and the guidance through the work. We would like to thank Andrew Booker, Jordan Ellenberg, Tom Fisher, Andrew Granville, Wei Ho, Bartosz Naskr\k{e}cki, Harald Schilly, Jeroen Sijsling, William Stein, Gonzalo Tornar\'{i}a and John Voight for their advice and help. The authors are indebted to Zev Klagsbrun for the rank computation of the curve with parameter $t=66/97$. The authors thank the organizers of the ``Curves and L-functions''
 summer school held at ICTP in 2017, where this project began: Tim Dokchitser, Vladimir Dokchitser, and Fernando Rodriguez Villegas. We used the open-source software SageMath and CoCalc extensively throughout this project.
Chan was supported by the European Research Council grant agreement No.\ 670239. Hanselman was supported by the research grant 7635.521(16) of the Science Ministry of Baden-W\"{u}rttemberg.

\section{Some preliminary properties of the \texorpdfstring{$(2,8)$}{(2,8)}-torsion family}\label{sec:properties}
In this section, we discuss the parametrization for the $(2,8)$-torsion family. 
We also show how the parameter height is related to the naive height and the conductor.

\subsection{The parametrization}\label{sec:par}
By expressing the torsion points explicitly, one can check that any curve with $\Z/2\Z \times \Z/8 \Z$-torsion can be described as an element of $\mathcal{F}$. Conversely, given a curve in $\mathcal{F}$, it is a straightforward calculation to verify that \[\left(\frac{2u}{(t+1)^2},\frac{4t(t^2 + 2t - 1)(t^2 + 1)}{(t + 1)^5(t - 1)^3}\right)\] is a point of order $8$. Hence the torsion subgroup is isomorphic to $\Z/2\Z \times \Z/8\Z$.

In each isomorphism class in $\mathcal{F}$, there are exactly $8$ different choices of $t$. We get these representatives using the transformations $t\mapsto -t$, $t\mapsto 1/t$ and $t\mapsto (1-t)/(1+t)$. We choose the $t$ corresponding to a curve with minimal height. The maps $t\mapsto -t$, $t\mapsto 1/t$ allow us to restrict $t=a/b$ to the range $(0,1)$. Assuming $a<b$, if $a \equiv b \equiv 1 \bmod 2$, the map $t\mapsto (1-t)/(1+t)$ allows us to take parameter $t'=a'/b'$, where $a'=(b-a)/2$ and $b'=(a+b)/2$. Then $t'$ would have a smaller height, since $a'<b'<b$. Thus, choosing $t=a/b\in (0,1)$ with $a$ and $b$ coprime with different parity, we get a unique representative for each isomorphism class. 

With this choice of parameter, we see that the number of curves with height $n$ is $\phi(n)$ if $n$ is even and $\phi(n)/2$ if $n$ is odd, where $\phi(n)$ is the Euler totient function. 
By \cite{Walfisz}, we have for any $\epsilon>0$, the estimate
\[\sum_{n\leq N}\phi(n)
=\frac{3}{\pi^2} N^2+O(N(\log N)^{2/3}(\log\log N)^{4/3}).\] 
Using the fact that $\phi(2n)$ is $\phi(n)$ if $n$ is odd and $2\phi(n)$ if $n$ is even, one can show that the total number of curves up to height $N$ is
\[\frac{2}{\pi^2}N^2+O(N(\log N)^{2/3}(\log\log N)^{4/3}).\]

\subsection{Naive height and parameter height}
Let $E$ be a curve given by the equation $y^2=x(x+1)(x+u^4)$ in $\mathcal{F}$ where $u = 2t/(t^2-1)$ and $t=a/b$ are chosen as above.
We show how the naive height and parameter height are related.

\begin{proposition}\label{prop:height}
Let $E/\Q$ be an elliptic curve in $\mathcal{F}$, with naive height $h$ and parameter height $H$. We have \[0.559\cdot h^{1/48} <H< 0.672\cdot h^{1/48}.\]
\end{proposition}

\begin{proof}
We start by giving a minimal Weierstrass model for our curve. Write $S=2ab$ and $T=b^2-a^2$, so $u=-S/T$. It follows that $S$ and $T$ are coprime where $S$ is even and $T$ is odd.
We write $E$ in short Weierstrass form
$y^2=x^3-Ax+B$, by putting
$A=27(S^8-S^4T^4+T^8)$ and $B=27 (S^4 - 2T^4) (2S^4-T^4) (S^4 + T^4)$.

One can check that there exists no prime $p$ such that $p^4\mid A$ and $p^6\mid B$, therefore this Weierstrass form is minimal.
With this, the naive height of $E$ is given by:
\[\begin{split}
h
%&=3^9\max\left\{4|S^8-S^4T^4+T^8|^3,(S^4 - 2T^4)^2 (2S^4-T^4)^2(S^4 + T^4)^2\right\}\\
&=3^9T^{24}\max\left\{4|1-u^4+u^8|^3,(1 - 2u^4)^2 (2-u^4)^2(1 + u^4)^2\right\}.
\end{split}\]
Since this expression is symmetric in $S$ and $T$, first assume $S<T$, so that $u\in (0,1)$.
Bounding the polynomials in $u$, we get
$3^{12}\cdot T^{24}/16\leq h\leq 4\cdot 3^9\cdot T^{24}$.
Note also, 
$\max(S,T)=\max(2ab,b^2-a^2)\in [2(\sqrt{2}-1)H^2,2H(H-1)]$.
Therefore $(\sqrt{2}-1)^{24}\cdot 3^{12}\cdot 2^{20}\cdot H^{48}< h< 3^9\cdot 2^{26}\cdot H^{48}$, which gives the result.
\end{proof}
\subsection{Size of the conductor}
\begin{figure}[ht]
\centering
\includegraphics[trim={2.5mm 2mm 2.5mm 2.5mm}, clip=true, width=1.0\textwidth]{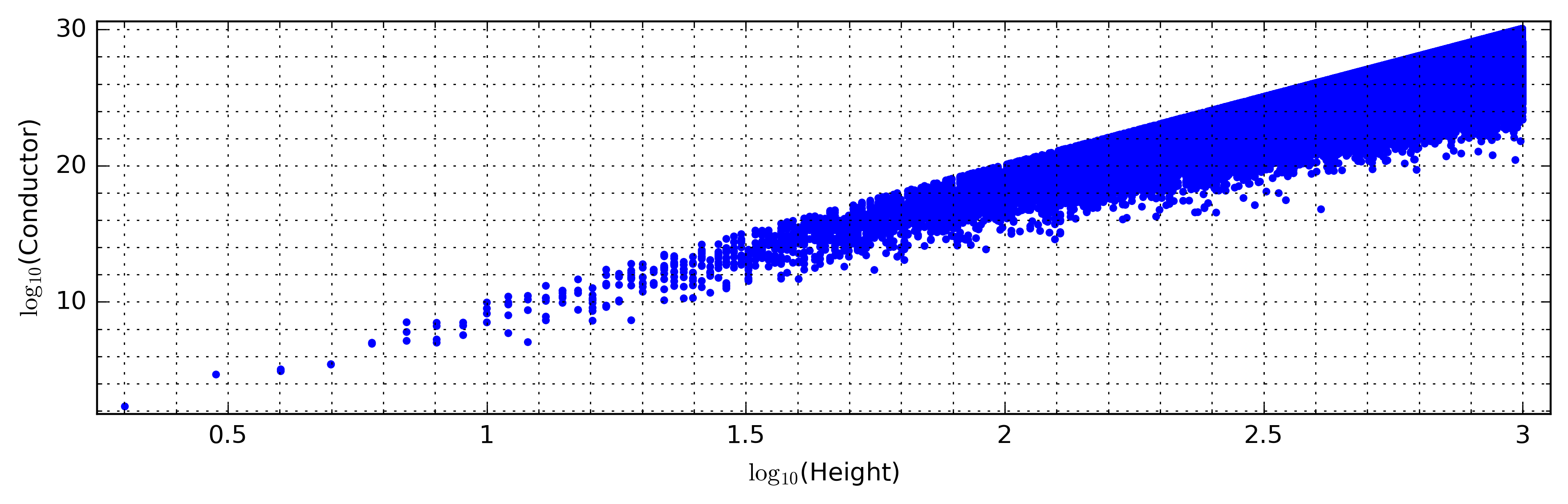}
\caption{Conductor of isomorphism classes in the $(2,8)$-torsion family.}\label{fig:conductor}
\end{figure}
Consider a curve in $\mathcal{F}$ with parameter $t=a/b$, where $a$ and $b$ are coprime and of different parity. This curve is isomorphic to
\[E:y^2=x(x+S^4)(x+T^4),\]
where $S=2ab$ and $T=b^2-a^2$ are coprime.
The discriminant of $E$ is $\Delta_E=16S^8T^8(T^4-S^4)^2$.
By Tate's algorithm \cite{Tate}, this curve has bad reduction precisely at the primes dividing $\Delta_E$, and the exponent of the conductor is always $1$. Therefore the conductor of $E$ is the product of primes dividing
\[ ab(b^2-a^2)(a^2 + b^2)(a^2-2ab-b^2) (a^2+2ab-b^2)=b^{10}t(1-t^2)(1+t^2)(t^2-2t-1)(t^2+2t-1). \]
The absolute value of the polynomial in $t$ is bounded from above in the interval $(0,1)$ by approximately $1.160$. 
Hence
$N(E)< 1.161\cdot H(E)^{10}$.

\section{Background}\label{sec:background}
Computing the rank of an elliptic curve over a number field is a difficult problem, and while there are a number of techniques that work well in practice, there is no known algorithm to carry this out in general. Here we review the main theorems and conjectures and discuss how they can be used to give conditional results.

\subsection{The BSD Conjecture}
The most famous conjecture on ranks of elliptic curves is the Birch and Swinnerton-Dyer Conjecture (BSD) \cite{BSD}.  Let $E$ be an elliptic curve defined over a number field with $L$-function  $L(s,E)$. The BSD Conjecture states that the rank of $E$ equals the order of vanishing of $L(s,E)$ at $s=1$, which is called the \emph{analytic rank} of $E$.
Assuming this conjecture allows us to obtain an upper bound of the rank from the $L$-function.

%The BSD conjecture has been proven to be true in some special cases. For example, the work of Gross--Zagier and Kolyvagin shows that the BSD conjecture holds for elliptic curves over $\mathbb{Q}$, if the analytic rank is equal to $0$ or $1$.

\subsection{The Minimalist Conjecture and Current Results}
It is believed that the root number, i.e.\ the sign of the functional equation of $L(s,E)$, is $1$ for half of all elliptic curves and $-1$ for the other half.
The Minimalist Conjecture, initially formulated by Goldfeld \cite{Goldfeld} for the quadratic twists families, states that with respect to any reasonable ordering, half of the elliptic curves have rank $0$ and half have rank $1$. This would mean the average rank should tend to $1/2$, and $0 \%$ of elliptic curves have rank at least $2$.
One of our main goals is to provide numerical evidence for this conjecture for the $(2,8)$-torsion family.

The following result of Bhargava and Shankar \cite{BhargavaShankar885} on the upper bound of the average rank of elliptic curves provides strong evidence for the Minimalist Conjecture.

\begin{theorem}[Bhargava--Shankar, \cite{BhargavaShankar885}]
The average rank of all elliptic curves over $\mathbb{Q}$ ordered by naive height is at most $0.885$.
\end{theorem}

\subsection{The Selmer Group and Descent}\label{sec:selmer}

For each integer $n \ge 2$, the $n$-Selmer group $\Sel_n(E)$ of $E$ over $\mathbb{Q}$ fits into an exact sequence of abelian groups
\begin{equation}\label{ses} 0 \rightarrow E(\mathbb{Q})/nE(\mathbb{Q}) \rightarrow \Sel_n (E) \rightarrow \Sha(E)[n] \rightarrow 0,\end{equation}
where $\Sha(E)[n]$ denotes the $n$-torsion subgroup of the Tate-Shafarevich group $\Sha(E)$ of $E$ over $\mathbb{Q}$. If $p$ is a prime, then $\Sel_p(E)$ is an elementary abelian $p$-group, whose dimension as an $\mathbb{F}_p$-vector space is called the $p$-Selmer rank of $E$, which is effectively computable and provides an upper bound on the rank via \eqref{ses}.

Explicitly, an element in the $n$-Selmer group of $E$ can be represented by a pair $(C,\pi)$, where $C$ is a genus $1$ curve which is locally soluble and $\pi$ is a map defined over $\mathbb{Q}$ that makes the following diagram commute:
\[\begin{tikzcd}
C  \arrow{rd}[above,right]{\pi}  \arrow{d}[left]{\simeq } \\
E \arrow{r}[below]{[n]}
& E
\end{tikzcd}\]
In this diagram, the vertical map $C \rightarrow E$ is an isomorphism defined over $\overline{\mathbb{Q}}$.
Determining (a lower bound for) the rank of $E$ is equivalent to finding rational points on $C$. If no rational point of $C$ can be found by a search by height, we apply the method of descent repeatedly.
More generally, given a rational isogeny $\phi : E \rightarrow E'$, there is a Selmer group associated to it, denoted as $\Sel_{\phi}(E)$. For the dual isogeny $\hat{\phi}:E'\rightarrow E$ of $\phi$, we denote the corresponding Selmer group as $\Sel_{\hat{\phi}}(E')$.
The following is a standard result, see for example \cite[Lemma 6.1]{SS}.
\begin{theorem}
Let $E$ and $E'$ be elliptic curves over $\Q$. Suppose there exists $\phi: E \rightarrow E'$ an isogeny of degree $2$. Then the following sequence is exact:
\[0\rightarrow E'(\Q)[\hat{\phi}]/\phi(E(\Q)[2])\rightarrow \Sel_{\phi}(E/\Q)\rightarrow
\Sel_2(E/\Q)\rightarrow \Sel_{\hat{\phi}}(E'/\Q).\]
\end{theorem}
For $E\in\mathcal{F}$, we have $\left|E'(\Q)[\hat{\phi}]/\phi(E(\Q)[2])\right|=1$, which implies that $$\left|\Sel_{\phi}(E/\Q)\right|\leq \left|\Sel_2(E/\Q)\right|.$$

Fisher \cite{Fisher} gives an efficient way to apply descent $6$ times on elliptic curves with full $2$-torsion structure. Moreover, since the $(2,8)$-torsion family has $\mathbb{Z}/2\mathbb{Z} \times \mathbb{Z}/8\mathbb{Z}$ torsion, there are two isogenous curves with full $2$-torsion structure. Applying Fisher's method to all three isogenous curves allowed us to determine the rank of more curves. Below is a picture of the isogenous curves and their torsion structures.
\begin{small}
\[\begin{tikzcd}[nodes={row sep=0.3cm}]
E \arrow{d}{} 
& E_{\textrm{tors}}(\Q)\cong 
\mathbb{Z}/2\mathbb{Z} \times \mathbb{Z}/8\mathbb{Z} 
\arrow{d}{} \\
E'  \arrow{d}{} 
& E'_{\textrm{tors}}(\Q)\cong 
\mathbb{Z}/2\mathbb{Z} \times \mathbb{Z}/4\mathbb{Z} 
\arrow{d}{} \\
E'' 
& E''_{\textrm{tors}}(\Q)\cong 
\mathbb{Z}/2\mathbb{Z} \times \mathbb{Z}/2\mathbb{Z} 
\end{tikzcd}\]
\end{small}

There are also a number of recent results on the size of Selmer groups:

\begin{theorem}[Bhargava--Shankar, \cite{BhargavaShankar}]
For $n \le 5$, the average size of $\Sel_n(E)$ for all elliptic curves $E/ \mathbb{Q}$ ordered by naive height is $\sigma(n)$, the sum of divisors of $n$.
\end{theorem}

The theorem implies that the average size of the $2$-Selmer group converges to $\sigma(2)=3$. However, this no longer holds for the family with nontrivial $2$-torsion.

\begin{theorem}[Klagsbrun--Lemke Oliver, \cite{KLO}]
The average size of $\Sel_2(E)$ is unbounded for the family of elliptic curves over $\Q$ with a torsion point of order $2$ ordered by a parameter height \footnote{The parameter height used here for an elliptic curve with a $2$-torsion point $E_{A,B}: y^2=x^3+Ax^2+Bx$, is $\max \{ |A|, B^2\}$.}.
\end{theorem}

Our data suggests that the average size of the $2$-Selmer group is also unbounded in the $(2,8)$-torsion family. In Section~\ref{subsec:selmerunbounded}, we give a proof of this fact.

\subsection{The Tamagawa Number}\label{sec:tam}
Let $E$ be an elliptic curve over $\Q$. 
The \emph{Tamagawa number} is the finite index
$c_p(E):=\#(E(\Q_p)/E_0(\Q_p))$,
where $E_0(\Q_p)$ is the subgroup of points in $E(\Q_p)$ which have good reduction. 
Each $c_p(E)$ can be easily computed from the coefficients of $E$ using Tate's algorithm \cite{Tate}.
The \emph{Tamagawa product} of $E$ is  
\[\mathcal{T}(E) = \prod_{p\leq \infty}c_p(E).\]
If there exists an isogeny $\phi: E \rightarrow E'$ of degree $2$, then the \emph{Tamagawa ratio} of $E$ is
\[ \mathcal{T}(E/E')=\frac{\left|\Sel_{\phi}(E)\right|}{\left|\Sel_{\hat{\phi}}(E')\right|}. \]

Consider the exact sequence induced by the isogeny $\phi$: 
\[\begin{tikzcd}[column sep = small]
0 \arrow[r]
& \ker(\phi) \arrow[r]
& E(\mathbb{Q}) \arrow[r, "\phi"]
& E'(\mathbb{Q}) \arrow[r, "\delta"]
& H^1(\mathbb{Q},\ker(\phi)) \arrow[r]
& H^1(\mathbb{Q},E) \arrow[r]
& \cdots.
\end{tikzcd}\]
Passing to a completion at a place $p$, we define $$H^1_{\phi}(\Q_p, \ker \phi) := \delta_p(E'(\mathbb{Q}_p)/\phi(E(\mathbb{Q}_p)) \subset H^1(\mathbb{Q}_p,\ker(\phi)).$$
Then the Tamagawa ratio can be related to the Tamagawa numbers as follows.
\begin{theorem}[Cassels, \cite{Cassels}, Lemma 3.1]
The Tamagawa ratio decomposes into a product of local factors as follows: 
\[ \mathcal{T}(E/E')= \prod_{p\leq\infty} \mathcal{T}_p(E/E'), \quad\text{where}\quad\mathcal{T}_p(E/E')=\frac{1}{2}\left|H^1_{\phi}(\Q_p, \ker \phi)\right|. \]
\end{theorem}

\begin{theorem}[Dokchitser--Dokchitser, \cite{localinvariants}, Lemma 4.2 and 4.3]
For $p \ne 2$ finite, \[\frac{1}{2}\left|H^1_{\phi}(\Q_p, \ker \phi)\right|= \frac{c_p(E')}{c_p(E)}.\]
\end{theorem}

\section{Computing ranks}\label{sec:computation}
\subsection{Enumerating curves}
We produce a list of all isomorphism classes in $\mathcal{F}$ up to height $N$ by computing the Farey sequence of order $N$ to get a list of $(a,b)$, where $a$ and $b$ are coprime and have opposite parities. Each pair $(a,b)$ gives a curve in $\mathcal{F}$ of minimal height in its isomorphism class. This gives us $\num{202462}$ ordered isomorphism classes of $(2,8)$-torsion curves with height less than $1000$.

\subsection{Procedure}
To make our rank computations feasible, we assume two standard conjectures: the Birch and Swinnerton-Dyer Conjecture (BSD) and the generalized Riemann hypothesis (GRH). BSD allows us to obtain an upper bound of the rank by computing the analytic rank numerically. GRH provides the conjecturally best bound for the error term of the $L$-function attached to an elliptic curve, which improves the efficiency of the analytic rank computation.
An immediate consequence of the BSD Conjecture is the Parity Conjecture, which states that the root number agrees with the parity of the rank. This allows us to determine the rank when the upper bound and lower bound we calculated for the rank differ by $1$.

We computed the rank using a combination of Sage \cite{sagemath} and Magma \cite{Magma}.  
We first ran Cremona's \texttt{mwrank} in Sage, which carries out $2$-descent and searches for rational points with low height. This function gave us an upper bound and a lower bound for the rank of each curve in our database. 
If the bounds agreed, this determined the rank. If the bounds differed by $1$, the rank is obtained conditional on the Parity Conjecture. This process determined the rank of $52.1\%$ of the curves. 

If the rank was not determined at this stage, we ran the \texttt{Sage} function \\ \texttt{analytic\_rank\_upper\_bound}, which computes an upper bound on the analytic rank conditional on GRH and takes a parameter $\Delta$, using Bober's method in \cite{Bober}. The runtime is exponential in $\Delta$, but a higher $\Delta$ potentially gives a better bound.
We ran the function repeatedly with increasing values of $\Delta$ up to at most $2.0$, or until the rank's upper bound differed from the lower bound by at most $1$. After this stage, we still had $44.2\%$ curves with unknown rank.

Because of the large number of curves remaining, it was computationally unfeasible to run with higher $\Delta$ for all of them. Restricting to curves with $H<100$, only $153$ remained at this stage, and we were able to continue the process up to $\Delta=3.8$. After this, only $15$ curves were left with $H<100$. Computing the analytic rank becomes more difficult as the conductor increases. Since the parameter height appears to be positively correlated with the conductor, as is seen in Figure~\ref{fig:conductor}, it became more and more difficult to determine the rank the further we got along. 

%For the remaining curves, we carried out further computations in Magma. 
Since our curves have full rational $2$-torsion, a recent implementation of Fisher's \\ \texttt{TwoPowerIsogenyDescentRankBound} \cite{Fisher} in Magma is faster and a better fit for our purposes. Using this, we were able to determine the ranks of more than $90\%$ of the curves up to $H<1000$.

For the remaining curves, we returned to Sage. We ran analytic rank with higher values of $\Delta$, up to at least $3.2$, and do a further point search using a higher bound in the \texttt{mwrank} function \texttt{two\_descent}. Altogether, the rank of $42.1\%$ of the curves in our database was determined purely via descent, hence unconditionally.

Initially there was one curve left with $H<100$: this is the curve with parameter $t=66/97$. Thanks to Klagsbrun for suggesting the use of \texttt{AnalyticRank} in Magma, we are able to show that this curve has rank $0$. The rank of all curves with $H<100$ are determined conditional on GRH and BSD.

%For curves with $H<100$, only one\footnote{In short Weierstrass form, this is
%$y^2 = x^3 - 19042627804923027301781026322193147x \\
%+ 1009379401557416277213540098882110665433479125641686$.
%This curve has root number $1$ and the rank is at most $2$ from our computation, so it should have rank $0$ or $2$.} was left: this is the curve with parameter $t=66/97$. Its rank is undetermined after running analytic rank with $\Delta=4.0$ and a point search with logarithmic height bound $15$. 
The list of high rank curves maintained by Dujella \cite{Dujella} contains $28$ rank $3$ curves, of which $26$ has $H<1000$. Our computations recovered the rank of $17$ of them. The rank of the remaining $9$ curves, which were all discovered by Fisher, were included in our database for completeness. In addition to the list, we found an extra rank $3$ curve at $t=9/296$. 

\section{Results and analysis of computed data}\label{sec:data}

\subsection{Rank}\label{subsec:avrank}

In the $(2,8)$-torsion family, we very quickly observe a possible turnaround point in average rank. The average rank seems to peak at $H=24$ with value $0.744$, after $121$ curves are computed, then steadily decreases to $0.626$ at $H=99$.

\begin{figure}[ht]
\centering
\includegraphics[trim={2.5mm 2mm 2.5mm 2.5mm}, clip=true, width=1.0\textwidth]{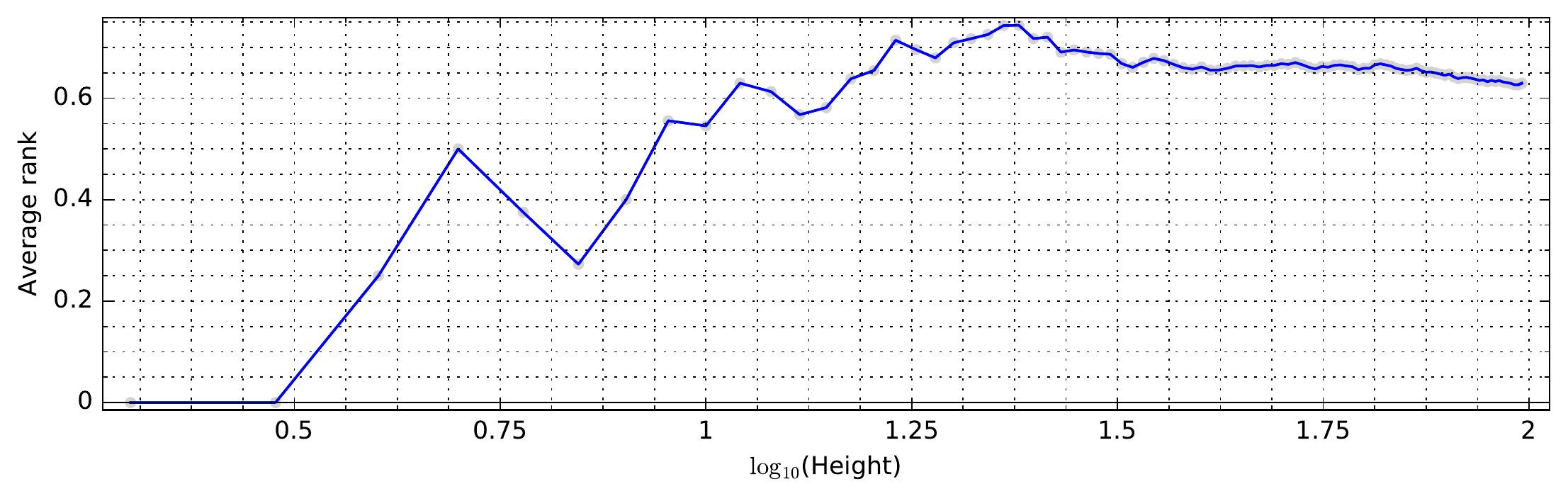}
\caption{Average rank up to height $100$ in the $(2,8)$-torsion family.}\label{fig:H100}
\end{figure}

Looking at all curves with $H<1000$, the behaviour is less certain because of the number of curves with undetermined ranks: we are only able to compute the rank of $\num{186718}$ curves which is $92.2\%$. For the remaining curves, we have upper bounds and lower bounds from our computations. None of these upper bounds is greater than $3$, so no rank $4$ curve can exist with $H<1000$. In Figure~\ref{fig:H1000}, we plot the computed upper and lower bounds for the average rank.

\begin{figure}[ht]
\centering
\includegraphics[trim={2.5mm 2mm 2.5mm 2.5mm}, clip=true, width=1.0\textwidth]{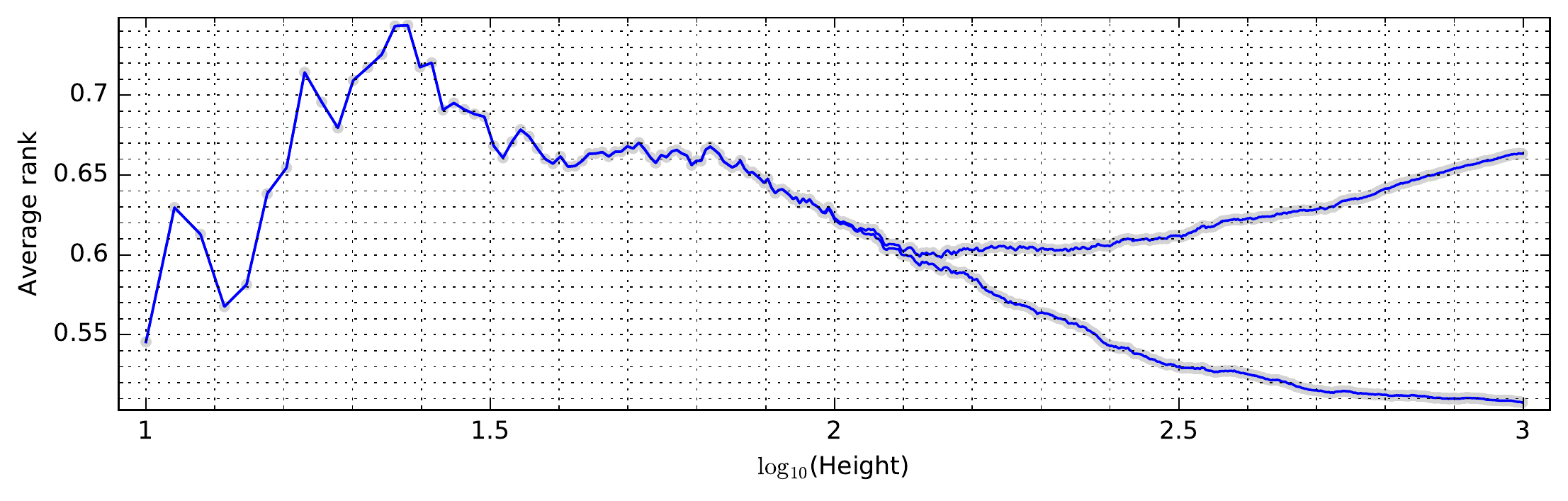}
\caption{Average rank up to height $1000$  in the $(2,8)$-torsion family.}\label{fig:H1000}
\end{figure}
\vspace{-.5em}
\begin{table}[ht]
\centering
\footnotesize
$\begin{array}{|l|r@{\hskip0pt} r|r@{\hskip0pt} r|r@{\hskip0pt} r|r@{\hskip0pt} r|}
\hline
\text{Rank}& H<100 & (\%) & H<250 & (\%) & H<500 & (\%) & H<1000 & (\%) \\ \hline
0 & 865 & (43.3) & 5672& (45.0) & 22143 & (43.8) & 84724& (41.8) \\ \hline
1 & 1021& (51.1) & 6243& (49.5) & 25108 & (49.7) & 101354 & (50.1) \\ \hline
2 & 111 & (5.6)& 298 & (2.4)& 445 & (0.9)& 613& (0.3)\\ \hline
3 & 3 & (0.2)& 10& (0.1)& 24& (0.0)& 27 & (0.0)\\ \hline
\geq 4 & 0 & (0.0)& 0 & (0.0)& 0 & (0.0)& 0& (0.0)\\ \hline
\text{Unknown} & 0 & (0.0)& 384 & (3.0)& 2845& (5.6)& 15743& (7.8)\\ \hline
\text{Total} & 2000& (100.0)& 12607 & (100.0)& 50565 & (100.0)& 202461 & (100.0)\\ \hline
\text{Average} & \multicolumn{2}{r|}{0.626} & \multicolumn{2}{r|}{[0.545, 0.606]} & \multicolumn{2}{r|}{[0.516, 0.628]} & \multicolumn{2}{r|}{[0.508,0.663]} \\ \hline
\end{array}$
\vspace{3pt} 
\caption{Rank distribution up to different heights.}
\end{table}

\subsection{Size of the \texorpdfstring{$2$}{2}-Selmer group}\label{subsec:avselmer}
To get a clearer picture of the behaviour of the average size of the $2$-Selmer group, we computed data beyond height $1000$, and it seems to be divergent. In Section~\ref{subsec:selmerunbounded}, we prove that this is indeed the case.
\begin{figure}[ht]
\centering
\includegraphics[trim={2.5mm 2mm 2.5mm 2.5mm}, clip=true, width=1.0\textwidth]{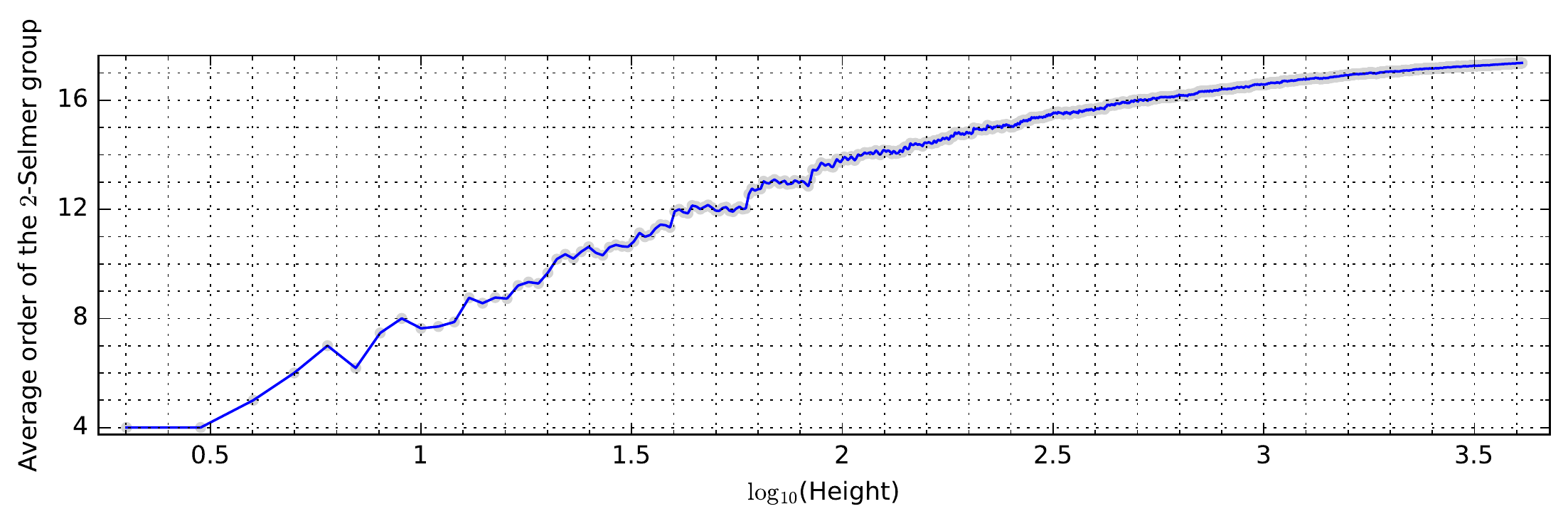}
\caption{Average size of the $2$-Selmer group in the $(2,8)$-torsion family.}
\end{figure}

\begin{table}[ht]
\centering
\footnotesize
$\begin{array}{|l|r@{\hskip3pt} r|r@{\hskip3pt} r|r@{\hskip3pt} r|r@{\hskip3pt} r|}
\hline
\rk \Sel_2(E) & H<100 & (\%) & H<1000 & (\%) & H<2000 & (\%) & H<4000 & (\%) \\ \hline
2 & 346 & (17.3) & 29943 & (14.8) & 117397 & (14.5) & 462688 & (14.3) \\ \hline
3 & 799 & (40.0) & 70856 & (35.0) & 278930 & (34.4) & 1107482 & (34.2) \\ \hline
4 & 586 & (29.3) & 62903 & (31.1) & 252357 & (31.1) & 1009839 & (31.2) \\ \hline
5 & 222 & (11.1) & 29287 & (14.5) & 120373 & (14.9) & 487277 & (15.0) \\ \hline
6 & 44 & (2.2) & 7934 & (3.9) & 34104 & (4.2) & 142043 & (4.4) \\ \hline
7 & 3 & (0.2) & 1386 & (0.7) & 6329 & (0.8) & 27823 & (0.9) \\ \hline
8 & 0 & (0.0) & 147 & (0.1) & 811 & (0.1) & 3743 & (0.1) \\ \hline
9 & 0 & (0.0) & 5 & (0.0) & 51 & (0.0) & 333 & (0.0) \\ \hline
10 & 0 & (0.0) & 0 & (0.0) & 3 & (0.0) & 28 & (0.0) \\ \hline
\geq 11 & 0 & (0.0) & 0 & (0.0) & 0 & (0.0) & 0 & (0.0) \\ \hline
\text{Total} & 2000 &(100) & 202461 &(100)  & 810352 &(100)  & 3241228 &(100)  \\ \hline
\text{Average }|\Sel_2(E)| & \multicolumn{2}{r|}{13.728} &\multicolumn{2}{r|}{16.574} &\multicolumn{2}{r|}{17.055} & \multicolumn{2}{r|}{17.361} \\ \hline
\end{array}$
\vspace{3pt} 
\caption{$2$-Selmer rank distribution up to different heights.}
\end{table}

\subsection{Tamagawa product}
The average Tamagawa product in the $(2,8)$-torsion family also behaves differently from the one in \cite{Main}. In their data, the average Tamagawa product peaks at $1.84$ at naive height $6.3\cdot 10^{5}$, then decreases with respect to the naive height. However in Figure~\ref{fig:tam}, we see that it is increasing in the $(2,8)$-torsion family, and that its value is much larger than $1.84$.
In Section~\ref{sec:avtam}, we show that the average Tamagawa product is unbounded for this family.
\begin{figure}[ht]
\centering
\includegraphics[trim={2.5mm 2mm 2.5mm 2.5mm}, clip=true, width=1.0\textwidth]{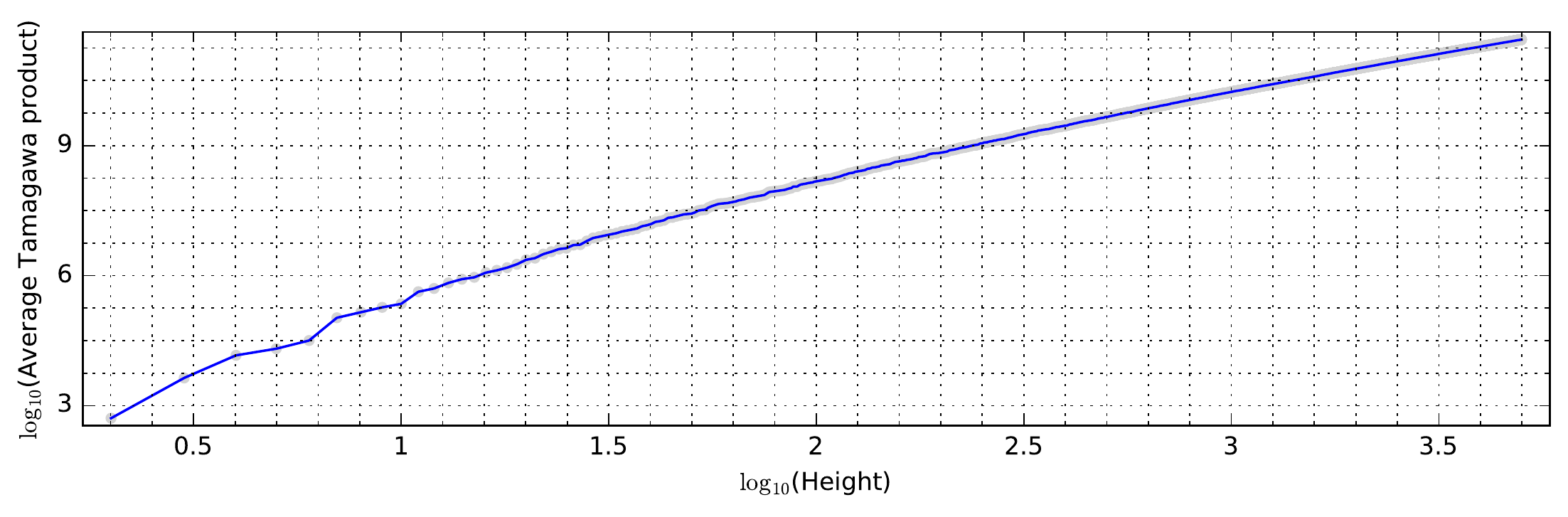}
\caption{Average Tamagawa product in $\log_{10}$ scale in the $(2,8)$-torsion family.}\label{fig:tam}
\end{figure}

\FloatBarrier
\subsection{Root number}
The average root number appears to converge to $0$, as shown in Figure~\ref{fig:rootnumber}. 
\begin{figure}[ht]
\centering
\includegraphics[trim={2.5mm 2mm 2.5mm 2.5mm}, clip=true, width=1.0\textwidth]{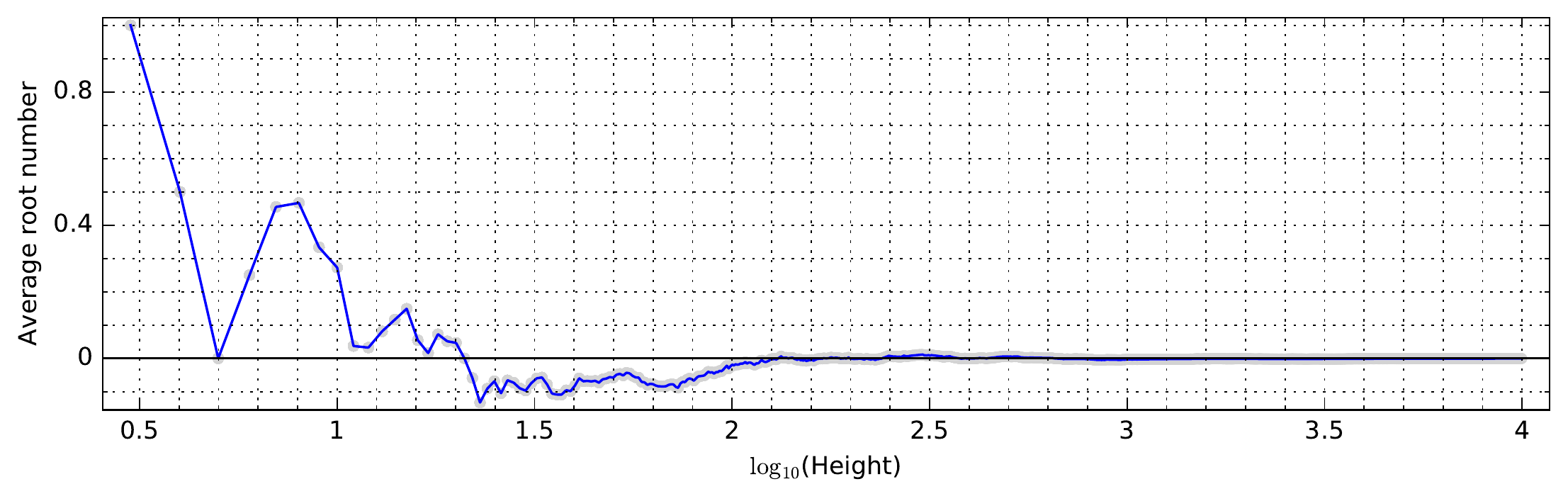}
\caption{Average root number in the $(2,8)$-torsion family.}\label{fig:rootnumber}
\end{figure}
\vspace{-.5em}
\begin{table}[ht]
\centering
\footnotesize
$\begin{array}{|l|r@{\hskip3pt} r|r@{\hskip3pt} r|r@{\hskip3pt} r|}
\hline
\text{Root number} & H<100 & (\%) & H<1000 & (\%) & H<10000 & (\%)  \\ \hline
1 & 976 & (48.8) & 100927 & (49.9) & 10125245 & (50.0)  \\ \hline
-1 & 1024 & (51.2) & 101534 & (50.1) & 10136574 & (50.0)  \\ \hline
\text{Total} & 2000 &(100) & 202461 &(100)  & 20261819 &(100)   \\ \hline
\text{Average} & \multicolumn{2}{r|}{-0.024000} & \multicolumn{2}{r|}{-0.002998} & \multicolumn{2}{r|}{-0.000559}  \\ \hline
  \end{array}$
  \vspace{3pt} 
\caption{Root number distribution up to different heights.}
\end{table}

\section{Proofs}\label{sec:proof}
\subsection{The average Tamagawa product is unbounded}\label{sec:avtam}

To find the numbers $c_p(E)$, we apply Tate's algorithm \cite{Tate}. We look at the model
\[E:y^2-xy=x^3+\frac{1}{4}\left(S^4+T^4-1\right)x^2+\frac{1}{16}S^4T^4x,\]
where $S=2ab$ and $T=b^2-a^2$. Again $a$ and $b$ are coprime and have opposite parities.
The discriminant of $E$ is $\Delta_E=\frac{1}{2^8}S^8T^8(T^4-S^4)^2$.
Note that $S$, $T$ and $(T^4-S^4)^2$ are pairwise coprime.
By Tate's algorithm \cite{Tate}, we get
\[
c_p = \begin{cases}
v_p(\Delta_E)&\text{if }  p\mid ST \text{ or }  \left( p\mid T^4-S^4 \text{ and } \leg{-1}{p}=1\right),\\
\hfil 2 &\text{if }  p\mid T^4-S^4 \text{ and } \leg{-1}{p}=-1,\\
\hfil 1& \text{otherwise}.
\end{cases}
\] 
Combining the local factors $c_p(E)$, we get 
\[\mathcal{T}(E)=\prod_{p}c_p(E)
=\prod_{\substack{p\mid T^4-S^4\\\leg{-1}{p}=-1}}2 \prod_{\substack{p^k\| (T^4-S^4)^2\\\leg{-1}{p}=1}}k \prod_{p^l\|\frac{1}{2^8} S^8T^8}l.\]

\begin{theorem}\label{thm:avtam}
The average Tamagawa product in the $(2,8)$-torsion family up to height $N$ has order of magnitude $(\log N)^{33}$.
\end{theorem}
\begin{proof}
We estimate the sum
\[S(N):=\sum_{\substack{a,b\leq N, 2\mid a\\(a,b)=1}} \prod_{\substack{p\mid T^4-S^4\\\leg{-1}{p}=-1}}2 \prod_{\substack{p^k\| (T^4-S^4)^2\\\leg{-1}{p}=1}}k \prod_{p^l\|\frac{1}{2^8} S^8T^8}l.\]
Let 
$H_1(a,b)=(a^2-b^2-2ab)(a^2-b^2+2ab)$, $H_2(a,b)=a^2+b^2$ and $H_3(a,b)=ab(b-a)(b+a)$.
Note that the factors $a^2-b^2-2ab,a^2-b^2+2ab,a^2+b^2, a,b,b-a,b+a$ are pairwise coprime.
Let 
\[f(H)=\prod_{\substack{p\mid H\\\leg{-1}{p}=-1}}2 \prod_{\substack{p^k\| H\\\leg{-1}{p}=1}}k\qquad \text{and}\qquad g(H)=\prod_{p^l\|H}l.\]
Let $P^+(x)$ and $P^-(x)$ denote the largest and smallest prime divisor of $x$ respectively.
Fix $\epsilon>0$.
Factorise $H_i(a,b)$ into $d_i$ and $H_i(a,b)/d_i$, so that $P^-(d_i)<N^{\epsilon}$, and $P^+(H_i(a,b)/d_i)\geq N^{\epsilon}$. Then  $\max_{a,b\leq N}\{H_1(a,b)^2H_2(a,b)^4,H_3(a,b)^8\}\leq N^{32}$, so $H_1(a,b)^2H_2(a,b)^4$ and $H_3(a,b)^8$ each has at most $32/\epsilon$ prime factors greater than $N^{\epsilon}$. Therefore
$f(d_1^2d_2^4)\leq f(H_1(a,b)^2H_2(a,b)^4)\ll_{\epsilon} f(d_1^2d_2^4)$. Similarly $g(d_3^8)\leq g(H_3(a,b)^8)\ll_{\epsilon} g(d_3^8)$.
We have
\[\begin{split}
S(N)
&=\sum_{\substack{a,b\leq N, 2\mid a\\(a,b)=1}} f(H_1(a,b)^2H_2(a,b)^4)g(H_3(a,b)^8)\\
&\asymp 
\sum_{\substack{d_1,d_2,d_3\\P^+(d_i)<N^{\epsilon}}}f(d_1^2d_2^4)g(d_3^8)
\sum_{\substack{a,b\leq N,\ 2\mid a,\ (a,b)=1\\ d_i\mid H_i(a,b)\\ P^-(\frac{H_i(a,b)}{d_i})\geq N^{\epsilon}}} 1.\end{split}\]
Write $a=\alpha+ud_1d_2d_3$ and $b=\beta+vd_1d_2d_3$.
Since $H_1$, $H_2$ and $H_3$ are pairwise coprime, we only need to look at coprime $d_1$, $d_2$ and $d_3$. Since $H_1$, $H_2$ are odd and $H_3$ is even, we consider only odd $d_1$, $d_2$ and even $d_3$.
Note that $a,b\mid H_3(a,b)$ by construction. Suppose $p\mid (a,b)$, then $p\mid d_2$ or $p>N^{\epsilon}$.
We have
\[
\sum_{\substack{a,b\leq N\\ 
\exists p\geq N^{\epsilon}:\ p\mid (a,b)}} 1
=O\left(\sum_{p\geq N^{\epsilon}} \left(\frac{N}{p}\right)^2\right)
=O(N^{2-\epsilon}).
\]
We can exclude pairs of $a$ and $b$ with $P^-((a,b))>N^{\epsilon}$ with a cost of $O(N^{2-\epsilon})$.
\[\sum_{\substack{a,b\leq N,\ 2\mid a,\ (a,b)=1\\ d_i\mid H_i(a,b)\\ P^-(\frac{H_i(a,b)}{d_i})\geq N^{\epsilon}}} 1
=
\sum_{\substack{\alpha,\beta<d_1d_2d_3\\ 2\mid \alpha,\  
d_i\mid H_i(\alpha,\beta)\\
p\mid d_1d_2d_3\Rightarrow (p\nmid \beta\text{ or }p\nmid \alpha)}}
\sum_{\substack{u,v<\frac{N}{d_1d_2d_3}\\
 P^-(\frac{H_i(a,b)}{d_i})\geq N^{\epsilon}}}
1
+O(N^{2-\epsilon}).\]
By the small sieve \cite[Theorem 2.6, p.85]{HR} we have
\[\sum_{\substack{u,v<\frac{N}{d_1d_2d_3}\\
 P^-(\frac{H_i(a,b)}{d_i})\geq N^{\epsilon}}}
1
\asymp
\frac{N^2}{d_1^2d_2^2d_3^2}\prod_{p< N^{\epsilon}}\left(1-\frac{7+\leg{-1}{p}+2\cdot \leg{2}{p}}{p}\right)
\asymp
\frac{N^2}{d_1^2d_2^2d_3^2(\log N)^7}.
\]
It remains to compute
\[\sum_{\substack{\alpha,\beta<d_1d_2d_3\\ 2\mid \alpha,\  
d_i\mid H_i(\alpha,\beta)\\
p\mid d_1d_2d_3\Rightarrow (p\nmid \beta\text{ or }p\nmid \alpha)}}1
=\sum_{\substack{\alpha,\beta<d_1\\  d_1\mid H_1(\alpha,\beta)\\
p\mid d_1\Rightarrow (p\nmid \beta\text{ or }p\nmid \alpha)}}1
\sum_{\substack{\alpha,\beta<d_2\\  d_2\mid H_2(\alpha,\beta)\\
p\mid d_2\Rightarrow (p\nmid \beta\text{ or }p\nmid \alpha)}}1
\sum_{\substack{\alpha,\beta<d_3\\ 2\mid \alpha,\ 
d_3\mid H_3(\alpha,\beta)\\
p\mid d_3\Rightarrow (p\nmid \beta\text{ or }p\nmid \alpha)}}1.\]
By the Chinese remainder theorem, it suffices to count the number of solutions of $H_i$ modulo $p^v\|d_i$ for each prime $p$ dividing $d_i$. We have
\begin{align*}
h_{1}(p^v)&:
=
\sum_{\substack{\alpha,\beta<p^v\\  
p^v\mid H_{1}(\alpha,\beta)\\
p\nmid \beta\text{ or }p\nmid \alpha}}1
=
\begin{cases}
4\phi(p^v) &\text{ if }2\text{ is a square modulo }p^v,\\
\hfil 0 &\text{ otherwise};
\end{cases}\\
h_{2}(p^v)&:
=
\sum_{\substack{\alpha,\beta<p^v\\  
p^v\mid H_{2}(\alpha,\beta)\\
p\nmid \beta\text{ or }p\nmid \alpha}}1
=
\begin{cases}
2\phi(p^v) &\text{ if }-1\text{ is a square modulo }p^v,\\
\hfil 0 &\text{ otherwise};
\end{cases}\\
h_3(p^v)&:=
\sum_{\substack{\alpha,\beta<p^v\\ 
p^v\mid H_3(\alpha,\beta)\\
p\nmid \beta\text{ or }p\nmid \alpha}}1
=
\begin{cases}
4\phi(p^v) &\text{ if }p\neq 2,\\
\hfil \phi(p^v) &\text{ if }p=2.\\
\end{cases}\end{align*}
We extend $h_1$, $h_2$ and $h_3$ to multiplicative functions.
Then the sum becomes
\[\begin{split}
S(N)&\asymp
\frac{N^2}{(\log N)^7}\sum_{\substack {d_1,d_2,d_3\\P^+(d_i)<N^{\epsilon}}}\frac{f(d_1^2d_2^4)g(d_3^8)h_1(d_1)h_2(d_2)h_3(d_3)}{d_1^2d_2^2d_3^2}\\
&\asymp\frac{N^2}{(\log N)^7}
\prod_{p<N^{\epsilon}}
\left(1+\frac{f(p^2)h_1(p)}{p^2}\right)
\left(1+\frac{f(p^4)h_2(p)}{p^2}\right)
\left(1+\frac{g(p^8)h_3(p)}{p^2}\right)\\
&\asymp\frac{N^2}{(\log N)^7}
\prod_{p<N^{\epsilon}}
\left(1+\frac{1}{p}\right)^4
\left(1+\frac{1}{p}\right)^4
\left(1+\frac{1}{p}\right)^{32}
\asymp 
N^2(\log N)^{33}.
\end{split}\]
The total number of curves up to height $N$ has order of magnitude
$N^2$ as discussed in Section~\ref{sec:par}. 
Therefore the average Tamagawa product is of the size $(\log N)^{33}$.
\end{proof}

\subsection{The average size of the \texorpdfstring{$2$}{2}-Selmer group is unbounded} \label{subsec:selmerunbounded}
We follow the approach in \cite{KLO} to show the average Tamagawa ratio diverges in the $(2,8)$-torsion family, which implies that the average size of the $2$-Selmer group is unbounded.

The curve obtained by the degree $2$ isogeny $\phi:E\rightarrow E'$ corresponding to the rational subgroup generated by the point $(0,0)$ is
\[E':y^2-xy=x^3+\frac{1}{4}\left(\left(S^2+T^2\right)^2+4S^2T^2-1\right)x^2+\frac{1}{4}\left(S^2T^2\left(S^2+T^2\right)^2\right)x,\]
which has discriminant $\Delta_{E'}=\frac{1}{2^4}S^4 T^4 (T^4 - S^4)^4$.
Using Tate's algorithm and looking at Table 1 in \cite{localinvariants}, we find that the Tamagawa ratio for any finite prime $p$ is
\[\mathcal{T}_p(E/E')
=\frac{c_p(E')}{c_p(E)}
=\begin{cases}
\hfil 2 &\text{ if }p\mid S^4-T^4\text{ and }\leg{-1}{p}=1,\\
\frac{1}{2} &\text{ if }p\mid ST,\\
\hfil 1 &\text{ otherwise}.
\end{cases}\]
Since the discriminants $\Delta_E$ and $\Delta_E'$ are both positive, we have $\mathcal{T}_{\infty}(E/E')=1$.

\begin{theorem}\label{thm:logtam}
The logarithmic Tamagawa ratio $t(a,b):=\log_2 \mathcal{T}(E/E')$ tends to a normal distribution with mean $-2\log\log N+O(1)$ and variance $6\log\log N+O(1)$.
\end{theorem}

Before we turn to the proof, let us look at the application of Theorem~\ref{thm:logtam}. We find that $t(a,b)\log 2$ tends to a normal distribution with mean $\mu:=-2(\log 2)(\log\log N)+O(1)$ and variance $\sigma^2:=6(\log 2)^2\log\log N+O(1)$.

Hence 
$\mathcal{T}(E/E')=\exp(t(a,b)\log 2)$ tends to a log-normal distribution which has mean $\exp(\mu+\frac{\sigma^2}{2})=e^{O(1)}(\log N)^{(3\log 2-2)\log 2}$.
Since $3\log 2-2>0$, the mean increases as $N$ increases.
From the discussion in Section~\ref{sec:selmer}, we know that $|\Sel_2(E)|\geq |\Sel_{\phi}(E)|\geq\mathcal{T}(E/E')$, so the following theorem is a corollary of Theorem~\ref{thm:logtam}.

\begin{theorem}\label{thm:avsel}
The average size of the $2$-Selmer group tends to infinity in the $(2,8)$-torsion family.
\end{theorem}

\begin{proof}[Proof of Theorem \ref{thm:logtam}]
Let $H_1=(a^2-b^2-2ab)(a^2-b^2+2ab)(a^2+b^2)$ and $H_2=ab(b-a)(b+a)$.
Throughout this proof, we will assume $p$ is an odd prime as the contribution of the prime $2$ can be taken into the error term.
Define 
\[
f_p(H):=\1_{p\mid H}\cdot \1_{\leg{-1}{p}=1}\quad \text{ and }\quad g_p(H):=\1_{p\mid H},\]
where $\1$ denotes the indicator function.
Then
\[t(a,b)=f(H_1(a,b))-g(H_2(a,b)), 
\text{ where } 
f(H):=\sum_p f_p(H)\text{ and }
g(H):=\sum_p g_p(H).\]
For any function $F$ and any property $\mathcal{P}$ defined on the set $\mathcal{A}_N:=\{(a,b): a,b\leq N,\ a\text{ and }b\text{ coprime and have opposite parities}\}$, define
\[\P_N(\mathcal{P})=\frac{\sum_{(a,b)\in\mathcal{A}_N}\1_{\mathcal{P}(a,b)}}{\left|\mathcal{A}_N\right|}\quad\text{ and }\quad
\E_N(F)=\frac{\sum_{(a,b)\in\mathcal{A}_N}F(a,b)}{\left|\mathcal{A}_N\right|}.
\]
Fix $\epsilon>0$. For $p\leq N^{\epsilon}$, by counting the number of solutions of $H_1,H_2$ modulo $p$,
\[\begin{split}
\E_N(f_p(H_1))&=\P_N (H_1\equiv 0\bmod p)=
\begin{cases}
\frac{6}{p+1}+O\left(\frac{1}{N^{2(1-\epsilon)}}\right)&\text{if }\leg{2}{p}=\leg{-1}{p}=1,\\
\frac{2}{p+1}+O\left(\frac{1}{N^{2(1-\epsilon)}}\right)&\text{if }\leg{2}{p}=-1,\ \leg{-1}{p}=1;\end{cases}\\
\E_N(g_p(H_2))&=\P_N (H_2\equiv 0\bmod p)=\frac{4}{p+1}+O\left(\frac{1}{N^{2(1-\epsilon)}}\right).\end{split}\]
Since $\max_{a,b\leq N}\{|H_1(a,b)|,|H_2(a,b)|\}\leq N^6$, each of $H_1$ and $H_2$ can only be divisible by at most $6/\epsilon$ prime factors larger than $N^{\epsilon}$, so $\sum_{p> N^{\epsilon}}f_p(H_1)$ and $\sum_{p> N^{\epsilon}}g_p(H_2)$ are bounded above by $6/\epsilon$.
Let $F(N):=\sum_{p\leq N^{\epsilon}}f_p(H_1)$ and $G(N):=\sum_{p\leq N^{\epsilon}}g_p(H_2)$. Then $F(N)=f(H)+O(1)$ and $G(N)=g(H)+O(1)$ for $(a,b)\in\mathcal{A}_N$.

We define the following random variables to model $f_p(H_1)$ and $g_p(H_2)$,
\[\begin{split}
X_p&=
\begin{cases}
1&\text{with probability }\frac{2}{p+1}\left(2+\leg{2}{p}\right)\\
0&\text{with probability }1-\frac{2}{p+1}\left(2+\leg{2}{p}\right)
\end{cases}\text{ if }\leg{-1}{p}=1;\\
Y_p&=\begin{cases}
1&\text{with probability }\frac{4}{p+1},\\
0&\text{with probability }1-\frac{4}{p+1},
\end{cases}\end{split}\]
and so that $\{X_p\}_p\cup\{Y_p\}_p$ are independent except $\P(X_p=1\text{ and }Y_p=1)=0$.
If $\leg{-1}{p}\neq 1$, $X_p=0$ with probability $1$.
Let $X(N)=\sum_{p\leq N^\epsilon} X_p$ and $Y(N)=\sum_{p\leq N^\epsilon} Y_p$.
By the multidimensional central limit theorem, $X(N)$ and $Y(N)$ converge to independent normal distributions as $N\rightarrow \infty$. Note that $X(N)$ has mean and variance $2\log\log N+O(1)$; $Y(N)$  has mean and variance $4\log\log N+O(1)$.

Since mixed moments determine the multinomial distribution, we want to show that the mixed moments of $F(N)$ and $G(N)$ converge to those of $X(N)$ and $Y(N)$. 
We have by construction
\[\begin{split}
\E_N\left(F(N)^{k}G(N)^l\right)
&=\sum_{\substack{p_1,\dots,p_k\leq N^{\epsilon}\\ q_1,\dots,q_l\leq N^{\epsilon}}}\P_N(H_1\equiv 0\bmod p_i\text{ and } H_2\equiv 0 \bmod q_j)\\
&=\E(X(N)^kY(N)^l)+O\left(\frac{(4\log\log N)^{k+l-1}}{N^{2(1-\epsilon)}}\right).
\end{split}\]
From this we compute
\[\begin{split}
&\E_N\left(\left(F(N)-\E_N(F(N))\right)^{k} \left(G(N)-\E_N(G(N))\right)^l\right)\\
&=\E\left(\left(X(N)-\E(X(N))\right)^k\left(Y(N
)-\E(Y(N))\right)^l\right)+O\left(\frac{(4\log\log N)^{k+l-1}}{N^{2(1-\epsilon)}}\right).
\end{split}\]
This shows that the distributions of $F(N)$ and $G(N)$ tend to those of $X(N)$ and $Y(N)$ respectively. The difference of two normal distribution is a normal distribution, hence $f(H_1)-g(H_2)=F(N)-G(N)+O(1)$ tends to a normal distribution with mean and variance as claimed.
\end{proof}

\bibliography{biblio}
\bibliographystyle{abbrv}
\end{document}